\documentclass[12pt]{amsart}
\usepackage[colorlinks=true,pagebackref,hyperindex,citecolor=green,linkcolor=red]{hyperref}
\usepackage{amsmath}
\usepackage{amsfonts}
\usepackage{amssymb}
\usepackage{color}
\usepackage[all]{xy}  
\usepackage{enumerate}
\usepackage[top=1in, bottom=1in, left=1in, right=1in]{geometry}
\usepackage{mathrsfs}

\usepackage{stmaryrd}

\theoremstyle{definition}
\newtheorem{theorem}{Theorem}[section]

\newtheorem{corollary}[theorem]{Corollary}
\newtheorem{lemma}[theorem]{Lemma}
\newtheorem{proposition}[theorem]{Proposition}

\theoremstyle{definition}
\newtheorem{definition}[theorem]{Definition}

\newtheorem{example}[theorem]{Example}

\newtheorem{remark}[theorem]{Remark}
\numberwithin{equation}{subsection}

\newcommand{\m}{\mathfrak{m}}

\newcommand{\NN}{\mathbb{N}}

\newcommand{\Spec}{\operatorname{Spec}}
\newcommand{\Depth}{\operatorname{depth}}
\newcommand{\Hom}{\operatorname{Hom}}

\newcommand{\Ext}{\operatorname{Ext}}

\newcommand{\Ker}{\operatorname{Ker}}

\newcommand{\Ass}{\operatorname{Ass}}

\newcommand{\Ht}{\operatorname{ht}}

\newcommand{\gs}{\geqslant}
\newcommand{\ds}{\displaystyle}

\newcommand{\p}{\mathfrak{p}}

\newcommand{\MIN}{\operatorname{Min}}

\newcommand{\ann}{\operatorname{ann}}
\newcommand{\ck}[1]{{#1}^{\vee}}
\newcommand{\n}{\mathfrak{n}}

\newcommand{\ps}[1]{\llbracket {#1} \rrbracket}
\newcommand{\GG}[1]{{#1}^{\Gamma}}

\begin{document}
\newcommand{\tens}{\otimes}
\newcommand{\hhtest}[1]{\tau ( #1 )}
\renewcommand{\hom}[3]{\operatorname{Hom}_{#1} ( #2, #3 )}

\title{A sufficient condition for strong $F$-regularity}
\author{Alessandro De Stefani}
\author{Luis N\'u\~nez-Betancourt}

\maketitle

\begin{abstract} Let $(R,\m,K)$ be an $F$-finite Noetherian local ring which has a canonical ideal $I \subsetneq R$. We prove that if $R$ is $S_2$ and $H^{d-1}_\m(R/I)$ is a simple $R\{F\}$-module, then $R$ is a strongly $F$-regular ring. In particular, under these assumptions, $R$ is a Cohen-Macaulay normal domain.
\end{abstract}
\section{Introduction}
Let $(R,\m,K)$ be a Noetherian local ring of positive characteristic and let $x\in\m$ be a nonzero divisor in $R$. A central question in the study of singularities is whether good properties of the ring $R/xR$ imply good properties of $R$. This is related to whether a type of singularity deforms. 
It is known  that $F$-purity and strong $F$-regularity, two important and well studied types of singularities in positive characteristic, do not deform. This was showed by an example of Fedder \cite{FedderFputityFsing} for $F$-purity, and by an example of Singh \cite{AnuragNotDeform}, for strong $F$-regularity. 
However, if $R$ is a Gorenstein ring, both $F$-purity and strong $F$-regularity do deform. 

Enescu  \cite{EnescuCovers} changed gears by looking at canonical ideals instead of ideals generated by nonzero divisors. Suppose that $R$ has a canonical ideal, i.e. an ideal $I$ such that $I\cong \omega_R$. Note that in a Gorenstein ring, an ideal is generated by a nonzero divisor if and only if it is a canonical ideal. Recently,
Ma \cite{MaRFmod} showed that, under mild assumptions on $R$, if $R/I$ is $F$-pure, then $R$ is also $F$-pure \cite[Theorem $3.4$]{MaFpurity}. Inspired by his result, we investigate if $R/I$ being strongly $F$-regular implies  that $R$ is strongly $F$-regular or, equivalently, if $R/I$ being F-rational implies that $R$ is strongly $F$-regular. Our main result is Theorem \ref{MainTech}, which is a more general version of the following.

\begin{theorem}\label{MainThm}
Let $(R,\m,K)$ be an excellent local ring of dimension $d$ and characteristic $p>0$. Suppose that $R$  is $S_2$ and it has a canonical ideal $I\cong \omega_R$ such that
$R/I$ is $F$-rational. Then $R$ is a strongly $F$-regular ring.
\end{theorem}

This theorem extends a result of Enescu \cite[Corollary 2.9]{EnescuCovers} by dropping the hypotheses of $R$ being a Cohen-Macaulay (normal) domain. We point out that these three conditions are implied by $R$ being strongly $F$-regular. 
In his work, Enescu uses properties of pseudo-canonical covers, while we focus on an interplay between Frobenius actions and $p^{-1}$-linear maps combined with structural properties of local cohomology.  As a consequence of Theorem \ref{MainThm}, we extend a result of Goto, Hayasaka, and Iai \cite[Corollary 2.4]{Goto} to rings which are not necessarily Cohen-Macaulay. Specifically, we show that, under the assumptions of Theorem \ref{MainThm},  if $R/I$  is regular, then $R$ is also regular (see Corollary \ref{ExtGoto}).

Throughout this article $(R,\m,K)$ will denote a Noetherian local ring of Krull dimension $d$ and  characteristic $p>0.$ $\ck{(-)}$  denotes the Matlis dual functor $\Hom_R(\ - \ ,E_R(K))$.
In addition, $\omega_R$ denotes a canonical module for $R$, which is a finitely generated $R$-module satisfying $\ck {\omega_R} \cong H^d_\m(R)$.
\section{Preliminaries}
\subsection{Canonical modules}
In this section we present several facts and properties regarding canonical modules over rings which are not necessarily Cohen-Macaulay. We refer to \cite{Aoyama,HoHuOmega,MaFpurity} for details.

We recall that not every ring has a canonical module; however, every complete ring has one. In fact, 
if $R$ is a homomorphic image of a Gorenstein local ring $(S,\n,L)$ of dimension $n$, then $\omega_R \cong \Ext^{n-d}_S(R,S)$.

\begin{proposition}[{\cite[Corollary $4.3$]{Aoyama}}]
Let $(R, \m,K)$ be
a local ring with canonical module $\omega_R.$ If $R$ is equidimensional, then for every prime ideal $P$,
$(\omega_R)_P$ is a canonical module for $R_P.$
\end{proposition}
\begin{proposition}[{\cite[Proposition 2.4]{MaFpurity}}]
Let $(R,\m,K)$ be a local ring with canonical module $\omega_R$. 
If $R$ is equidimensional and unmixed, then the following conditions are equivalent:  
\begin{enumerate}[(1)]
\item $\omega_R$ is isomorphic to an ideal $I \subseteq R$.
\item $R$ is generically Gorenstein, i.e. if $R_\p$ is Gorenstein for all $\p \in \MIN(R) (=\Ass(R))$.
\item $\omega_R$ has rank $1$.
\end{enumerate}
Moreover, when any of these equivalent conditions hold, $I$ is a height one ideal containing a nonzero divisor of $R$, and $R/I$ is equidimensional and unmixed \cite[Proposition 2.6]{MaFpurity}. If, in addition, $R$ is Cohen-Macaulay, then $R/I$ is Gorenstein.
\end{proposition}

\begin{definition}
Let $k$ be a positive integer. Recall that a finitely generated $R$-module $M$ is said to satisfy {\it Serre's condition $S_k$} (or simply $M$ is $S_k$) if
\[
\Depth(M_\p) \gs \min\{k,\Ht(\p)\}
\]
for all $\p \in \Spec(R).$
\end{definition}
$R$ is $S_k$ if it satisfies Serre's condition $S_k$ as a module over itself. If $R$ is $S_2,$ then $R$ is  unmixed. Furthermore, when $(R,\m,K)$ is local and catenary (e.g. when it is excellent), the $S_2$ condition also implies that $R$ is equidimensional. If $R$ is excellent and $S_2$, then its $\m$-adic completion $\widehat{R}$ is also $S_2$, and if $I$ is a canonical ideal of $R$, then $I \widehat R = \widehat I$ is a canonical ideal of $\widehat R$.




\subsection{Methods in positive characteristic}
We recall some of the definitions of singularities for rings of positive characteristic. We refer the interested reader to \cite{CraigBookTC,KarenSchoolVanishing,SurveyTestIdeals,SurveyP-1maps}  for surveys and a book on these topics.\\ 

Let $(R,\m,K)$ be a Noetherian local ring of characteristic $p>0,$ and let $F^e:R \to R$ be the $e$-th iteration of the Frobenius endomorphism on $R$, where $e$ is a positive integer. Let $M$ be an $R$-module. By $F_*^e(M)$ we denote $M$ viewed as a module over $R$ via the action of $F^e$. Specifically, for any $F_*^e(m_1),F_*^e(m_2) \in F_*^eM$ and for any $r \in R$ we have
\[
F_*^e(m_1) + F_*^e(m_2) = F_*^e(m_1+m_2) \ \ \ \mbox{ and } \ \ \ r \cdot F_*^e(m_1) = F_*^e(r^{p^e} m_1).
\]
If $e=1,$ we omit $e$ in the notation.
When $R$ is reduced, the endomorphism $F^e$ can be identified with the inclusion of $R$ into $R^{1/p^e}$, the ring of its $p^{e}$-th roots.

\begin{definition}
 $R$ is called {\it $F$-finite} if $F_*(R)$ is a finitely generated $R$-module.  
\end{definition}

A local ring $(R,\m,K)$ is $F$-finite if and only if  it is excellent and $[K:K^p] < \infty$  \cite{F-finExc}.

\begin{definition} 
 $R$ is  {\it $F$-pure} if  $F \otimes 1_M : R \otimes_R M \to R \otimes_R M$ is injective for all $R$-modules $M$. 
$R$ is  {\it $F$-split} if  the map $R  \to F_* R$ splits. 
\end{definition}

\begin{remark}\label{FpureFsplit}
If $R$ is an $F$-pure ring, $F$ itself is injective and $R$ must be a reduced ring. We have that $R$ is $F$-split if and only if $R$ is a direct summand of $F_* R$. If $R$ is an $F$-finite ring, $R$ is $F$-pure if and only $R$ is $F$-split \cite[Lemma $5.1$]{HoRo}.
As a consequence, if $(R,\m,K)$ is $F$-finite, we have that $R$ is $F$-pure if and only if $\widehat{R}$ is $F$-pure.
If $R$ is $F$-finite, we use the word $F$-pure to refer to both.
\end{remark}
\begin{definition}[\cite{BB-CartierMod}]
We say that an additive map $\phi:M\to M$ is {\it $p^{-e}$-linear} if  $\phi(r^{p^e}v)=r\phi(v)$ for every $r\in R, v\in M.$ 
\end{definition}

There is a bijective correspondence between $p^{-1}$-linear maps on $M$ and  $R$-module homomorphisms $F_*M \to M$.

\begin{definition}
We define the ring $R\{F\}$ as 
$\frac{R\langle F\rangle}{R(r^pF-Fr| r\in R)}$, the non-commutative $R$-algebra generated by $F$ with relations $r^p\cdot F=F\cdot r$, for $r \in R$.
\end{definition}
\begin{definition}
We say that an $R$-module  $M$ has a {\it Frobenius action},  if there is an additive map $F:M \to M$ such that $F(ru) = r^pF(u)$ for $u \in M$ and $r \in R$. 
\end{definition}

There is a natural equivalence between $R\{F\}$-modules and $R$-modules with a Frobenius action. In addition, every Frobenius action of $M$ corresponds to an $R$-module homomorphism $M\to F_*M$. If  $R$ is complete and $F$-finite, then
 $(F_* M)^{\vee} \cong F_*( M^\vee)$ \cite[Lemma $5.1$]{BB-CartierMod}. Then, there is an induced map
\[
F_*(M^\vee)\cong (F_* M)^{\vee}\to (M)^{\vee},
\]
which gives a correspondence between Frobenius actions on $M$ and  $p^{-1}$-linear maps on $\ck{M}$.
In this case, we have that the Frobenius map $F: R\to R$ induces a Frobenius action on $H^d_\m(R).$ Suppose $R$ has a canonical module $\omega_R.$ Let $\Phi:\omega_R\to \omega_R$ be the $p^{-1}$-linear map corresponding to the Matlis dual of $F:H^d_\m(R)\to H^d_\m(R)$.
We will refer to $\Phi$ as the {\it trace map} of $\omega_R$.

\begin{definition} \label{DefFRational}
A local ring $(R,\m,K)$ of dimension $d$ is called {\it $F$-rational} if it is Cohen-Macaulay and $H^d_\m(R)$ is simple as a $R\{F\}$-module. 
\end{definition}

We point out that this is not the original definition introduced by Hochster and Huneke \cite{HoHu1}, which is in terms of tight closure: $R$ is $F$-rational if the ideals generated by parameters are tightly closed.  However, both definitions are equivalent due to Smith \cite[Theorem $2.6$]{KarenFrational}. $F$-rational local rings have nice singularties; for example, they are normal domains.

\begin{definition}[\cite{HoHu1}]\label{DefFfinStrFreg} 
An $F$-finite ring $R$ is  {\it strongly $F$-regular} if for all nonzero elements $c \in R$, the $R$-linear homomorphism $\varphi:R \to F_*^e(R)$ defined by $\varphi(1) =  F_*^e(c)$ splits for $e \gg 0$.
\end{definition}

\begin{theorem}[{\cite[Theorem 3.1 c)]{HoHuStrong}}]\label{HoHu89}
Every $F$-finite regular ring is strongly $F$-regular. 
\end{theorem}

It is a well-known fact that strong $F$-regularity implies $F$-rationality. This could be seen from the relation that these notions have with tight closure (see \cite{HoHu1}).

\begin{definition}[\cite{KarlCentersFpurity}]
Suppose that $R$ is an integral domain. The \emph{test ideal} $\tau(R) \subseteq R$ is defined as the smallest non-zero compatible ideal of $R$. 
\end{definition}

$\tau(R)$ is the big test ideal originally defined by Hochster and Huneke \cite{HoHu1} in terms of tight closure. Schwede \cite[Theorem $6.3$]{KarlCentersFpurity} proved that the definition above is equivalent. We have that $R$ is strongly $F$-regular if and only if $\tau(R) = R$. Furthermore, $\tau(R) = \ann_{E_R(K)} 0^*_{E_R(K)}$, where $0^*_{E_R(K)}$ denotes the tight closure of $0$ in the injective hull $E_R(K)$ of $K$. We have that $R$ is strongly $F$-regular if and only if $0^*_{E_R(K)}= 0$.

\begin{remark} 
One can define strongly $F$-regular rings for rings that are not $F$-finite by requiring that for all nonzero elements $c \in R$, the $R$-linear homomorphism $\varphi:R \to F_*^e(R)$ defined by $\varphi(1) =  F_*^e(c)$ is pure for $e \gg 0$. If $R$ is $F$-finite, this definition is equivalent to Definition \ref{DefFfinStrFreg}. Furthermore, $R$ is strongly $F$-regular if and only if $\tau(R) = \ann_{E_R(K)} 0^*_{E_R(K)} = R$, as in the $F$-finite case \cite[Theorem $7.1.2$]{SmithThesis}.
Regular rings are strongly $F$-regular also for non $F$-finite rings. This can be proven using the Gamma construction \cite[Discussion 6.11 \& Lemma 6.13]{HoHu2} and Aberbach and Enescu's results on base change for test ideals \cite[Corollary 3.8]{AEBaseChange}.
\end{remark}

\begin{remark}\label{Rem Subset Cartier}
Let $\phi:R\to R$ be a $p^{-e}$-linear map. For every ideal $J\subseteq R$ we have $\phi(J^{[p^e]}) \subseteq J$. If $\phi$ is surjective then equality holds. Furthermore, if $\phi$ is surjective then $R$ is $F$-pure. In fact, there exists an element $r\in R$ such that $\phi(r)=1$, and then the $R$-linear homomorphism $\varphi:F^e_* R\to R$ defined by $\varphi(F_*^ex)=\phi(rx)$ gives the desired splitting.
\end{remark}

\begin{definition}[\cite{KarlCentersFpurity}]
Let $\phi:R \to R$ be a $p^{-e}$-linear map and let $J \subseteq R$ be an ideal. $J$ is  $\phi$-compatible if $\phi(J) \subseteq J$. An ideal $J$ is said to be compatible if it is $\phi$-compatible for all $p^{-e}$-linear maps $\phi:R\to R$ and all $e\in \NN$. 
\end{definition}

Compatible ideals in the definition above were used previously in different contexts \cite{MR85,KarenFrational,LyuKaren,HaraTakagi}. 
\begin{remark}\label{Principal}
If $R$ is a Gorenstein $F$-finite ring, we have that
$\Hom_R(F^e_*R,R)\cong F^e_* R$ as  $F^e_* R$-modules, and the isomorphism is given by precomposition with multiplication by elements in $F^e_* R.$  
Let $\Phi$ be the $p^{-1}$-map corresponding to a generator of  $\Hom_R(F_*R,R)\cong F^e_* R$ as a $F_* R$-module.
If an ideal $J$ is $\Phi$-compatible, then it is compatible
\cite[Theorem $3.7$]{SurveyTestIdeals}.
\end{remark}

\section{Canonical ideals and strong $F$-regularity}
We start by proving preparation lemmas that imply that under the hypotheses of Theorem \ref{MainThm}, $R$ is $F$-pure.
\begin{lemma}\label{Lemma F-inj}
Let $(R,\m,K)$ be a Noetherian local ring of dimension $d$. Suppose that $R$ has a canonical ideal $I \subsetneq R$ such that $H^{d-1}_\m(R/I)$ is a simple $R\{F\}$-module. Then the Frobenius map $F:H^{d-1}_\m(R/I)\to H^{d-1}_\m(R/I)$ is injective.
\end{lemma}
\begin{proof} 
Since the Frobenius action on $\Ker(F) \subseteq H^{d-1}_\m(R/I)$ is trivial, we have that $\Ker(F)$ is an $R\{F\}$-submodule. Then either $\Ker(F)=0$ or $\Ker(F)=H^{d-1}_\m(R/I)$, because $H^{d-1}_\m(R/I)$ is simple.  If $\Ker F=H^{d-1}_\m(R/I)$, then every $R$-submodule of $\Ker F$ is an $R\{F\}$-module.  Since $H^{d-1}_\m(R/I)$ is a simple $R\{F\}$-module, $\Ker F=H^{d-1}_\m(R/I)$ must be a simple $R$-module, that is $H^{d-1}_\m(R/I)\cong R/\m = k$. Since $\dim(R/I)=d-1,$ this is only possible if $d=1$ and $R/I$ is zero-dimensional. However, if $\dim(R/I)=0$, we have $H^{d-1}_\m(R/I)=R/I= R/\m$ and $\Ker(F)=0$ in this case.
\end{proof}

\begin{remark}\label{Rem SES}
Suppose that $ H^{d-1}_\m(R/I)$ is a simple $R\{F\}$-module. From the short exact sequence $0\to I\to R\to R/I\to 0,$ we obtain an exact sequence of $R\{F\}$-modules
\[
\xymatrixcolsep{5mm}
\xymatrixrowsep{2mm}
\xymatrix{
H^{d-1}_\m(R/I) \ar[r] & H^d_\m(I) \ar[r] & H^d_\m(R) \ar[r] & 0.
}
\]
The map $H^{d}_\m(I)\to H^{d}_\m(R)$ is not injective by \cite[Lemma $3.3$]{MaFpurity}, and its kernel is a non-zero $R\{F\}$-submodule of $H^d_\m(I)$. Since $H^{d-1}_\m(R/I)$ is a simple $R\{F\}$-module, the first map in the sequence above must be injective. Hence, when $H^{d-1}_\m(R/I)$ is a simple $R\{F\}$-module, we have a short exact sequence of $R\{F\}$-modules
\[
\xymatrixcolsep{5mm}
\xymatrixrowsep{2mm}
\xymatrix{
0\ar[r] & H^{d-1}_\m(R/I) \ar[r] &  H^{d}_\m(I)\ar[r] &  H^{d}_\m(R)\ar[r] & 0.
}
\]
\end{remark}
\begin{remark} \label{completion} 
Note that $H^{d-1}_\m(R/I)$ and $H^d_\m(R)$ are simple $R\{F\}$-modules if and only if $H^{d-1}_{\widehat{\m}}(\widehat{R}/\widehat{I})$ and $H^d_{\widehat{\m}}(\widehat{R})$ are simple $\widehat{R}\{F\}$-modules.
\end{remark}
\begin{lemma} \label{dual} 
Let $(R,\m,K)$ be an excellent local ring of dimension $d$. Suppose that $R$ is equidimensional and unmixed, and it has a canonical ideal $I\subsetneq R$. If $H^{d-1}_\m(R/I)$ is a simple $R\{F\}$-module, then $H^{d-1}_\m(R/I) \cong \ck{(R/I)}$.
\end{lemma}
\begin{proof} By Remark \ref{completion} we can assume that $R$ is complete. By Remark \ref{Rem SES} we have a short exact sequence of $R\{F\}$-modules $0 \to H^{d-1}_\m(R/I) \to H^d_\m(I) \to H^d_\m(R) \to 0$. Taking the Matlis dual we get an exact sequence of $R$-modules:
\[
\xymatrixcolsep{5mm}
\xymatrixrowsep{2mm}
\xymatrix{
0\ar[r] & \ck{H^d_\m(R)} \cong J \ar[r] & \ck{H^d_\m(I)} \cong R \ar[r] & \ck{H^{d-1}_\m(R/I)} \ar[r] & 0,
}
\]
where $J \cong \omega_R$ is potentially another canonical ideal for $R$. We then get that $\ck{H^{d-1}_\m(R/I)} \cong R/J =:\omega_{R/I}$ is a canonical module for $R/I$, and we want to show that $J=I$. We have a homomorphism of $R/I$-modules:
\[
R/I \longrightarrow \Hom_{R/I}(\omega_{R/I},\omega_{R/I}) \cong \Hom_{R/I}(R/J,R/J) \cong R/J,
\]
which is just the map induced by the inclusion $I\subseteq J$. Since $R$ is equidimensional and unmixed, so is $R/I$, and the kernel of the above map is trivial \cite{HoHuOmega}. In addition, this kernel is $J/I$.  Therefore, $J=I$ and $H^{d-1}_\m(R/I) \cong \ck{(R/I)}$.
\end{proof}

\begin{proposition}\label{PropFpure}
Let $(R,\m,K)$ be an $F$-finite Noetherian local ring of dimension $d$. Suppose that $R$ is $S_2$  and it has a canonical ideal $I\subsetneq R$ such that $H^{d-1}_\m(R/I)$ is a simple $R\{F\}$-module. Then, $R/I$ is an $F$-pure ring. As a consequence, $R$ is an $F$-pure ring.
\end{proposition}
\begin{proof}
By Remarks \ref{completion} and \ref{FpureFsplit}, we may assume that $R$ is a complete ring.
We have that the Frobenius action on $H^{d-1}_\m(R/I)$ is injective by Lemma \ref{Lemma F-inj}. This induces a surjective $p^{-1}$-linear map on $R/I=\left( H^{d-1}_\m (R/I)\right)^\vee.$ Then, $R/I$ is $F$-pure by Remark \ref{Rem Subset Cartier}.
Therefore, $R$ is also $F$-pure \cite[Theorem $3.4$]{MaFpurity}.
\end{proof}

The simplicity of $H^{d-1}_\m(R/I)$ forces $R/I$ to have mild singularities, as we show in the following result. This result will be needed in the proof of Theorem \ref{MainTech}.

\begin{theorem} \label{canonicalFreg} Let $(R,\m,K)$ be a Noetherian local $F$-finite ring of dimension $d$. Suppose that $R$ is equidimensional and unmixed, and that it has a canonical ideal $I\subsetneq R$ such that $H^{d-1}_\m(R/I)$ is a simple $R\{F\}$-module. Then $R/I$ is a  strongly $F$-regular Gorenstein ring.
\end{theorem}
\begin{proof}
We note that $R$ is strongly $F$-regular and Gorenstein if and only if its completion is also strongly $F$-regular and Gorenstein. 
We can assume that $R$ is complete
by Remark \ref{completion}. 
By Lemma \ref{dual}, it follows that $\ck{H^{d-1}_\m(R/I)} = \omega_{R/I} \cong R/I$. To prove that $R/I$ is Gorenstein, it remains to show that it is Cohen-Macaulay. Then, it suffices to show that $R/I$ is strongly $F$-regular. Let $\Phi:R/I \to R/I$ be the $p^{-1}$-linear map which is dual to the Frobenius action on $H^{d-1}_\m(R/I)$. We note that $\Phi$ is surjective by Lemma \ref{Lemma F-inj}.  
Let $c\in R/I$ be a nonzero element and set $J:=\bigcup_{e\in\NN}\Phi^e(c (R/I))$. We have  that $J$ is an ideal of $R/I$ that contains $c.$ In addition,  $\Phi(J)=J$ by Remark \ref{Rem Subset Cartier} and Lemma \ref{Lemma F-inj}. Since $c\neq 0$, $J$ is a nonzero ideal compatible with $\Phi$, and thus $\ck J$ corresponds to a nonzero $R\{F\}$-submodule of $ H^{d-1}_\m(R/I)$. But the latter is a simple $R\{F\}$-module, therefore $\ck J =H^{d-1}_\m(R/I)$, and hence $J=R/I$. In particular, there exists an element $r\in R/I$ and an integer $N$ such that $\Phi^N(rc)=1$. Let $\varphi:F^N_*(R/I)\to R/I$ be the $R/I$-linear map defined by 
$\varphi(F^N_* x)=\Phi^N(rx)$ for all  $x \in R/I.$
We have that $\varphi(F^e_* c)=1$. Since $0 \ne c \in R/I$ was chosen arbitrarily, we conclude that $R/I$ is strongly $F$-regular. 
\end{proof}

We recall a result of Goto, Hayasaka, and Iai \cite{Goto} that is needed to prove Theorem \ref{MainThm}.
\begin{proposition}[{\cite[Corollary $2.4$]{Goto}}] \label{Goto} Let $(S,\m,K)$ be a Cohen-Macaulay local ring which has a canonical ideal $I \subsetneq S$ such that $S/I$ is a regular local ring. Then $R$ is regular.
\end{proposition}

Now, we are ready to prove our main theorem.

\begin{theorem}\label{MainTech}
Let $(R,\m,K)$ be a Noetherian local $F$-finite ring of dimension $d$ and characteristic $p>0$. Suppose that $R$  is $S_2$ and it has a canonical ideal $I\subsetneq R$ such that
$H^{d-1}_\m(R/I)$ is a simple $R\{F\}$-module. Then $R$ is a strongly $F$-regular ring.
\end{theorem}

\begin{proof} Under our assumptions on $R$, we have that $\tau(\widehat R) \cap R = \tau(R)$  \cite[Theorem $2.3$]{LyuKaren}
and $I\widehat{R}$ is a canonical ideal for $\widehat{R}.$ Thus, it suffices to prove our claim assuming that $R$ is complete by Remark \ref{completion}. By Theorem \ref{canonicalFreg}, $I$ is a height one prime ideal. Since $R$ is equidimensional, we have that $(\omega_R)_I$ is a canonical module for $R_I$, and in particular $IR_I$ is a canonical ideal for $R_I$. Since $R_I$ is a one-dimensional Cohen-Macaulay local ring and $R_I/IR_I$ is a field, we have that $R_I$ must be regular by Proposition \ref{Goto}.

We finish proving the theorem by means of contradiction. Assume that $R$ is not strongly $F$-regular.  Let $\tau(R)$ be the test ideal of $R$. We claim that $\tau(R) \subseteq I$. Let $N=\ann_{E_R(K)} \tau(R)$, which is a submodule of $E_R(K) \cong H_\m^d(I)$  compatible with every Frobenius action on $E_R(K)$ (see \cite{LyuKaren}). 
In particular, $N$ is an $R\{F\}$-submodule of $H_\m^d(I)$. As $H_\m^{d-1}(R/I)$ is a simple $R\{F\}$-submodule of $H_\m^d(I)$, it must be contained in $N$. In fact, they cannot be disjoint because they both contain the socle of $H_\m^d(I)$. Taking annihilators in $R$ and applying Matlis duality, we get
\[
\ds \tau(R) = \ann_R(N) \subseteq \ann_R(H^{d-1}_\m(R/I)) = I.
\]
Since the test ideal defines the non-strongly $F$-regular locus \cite[Theorem 7.1]{LyuKaren}, $R_I$ is not strongly $F$-regular. This is a contradiction because every regular ring is strongly $F$-regular by Theorem \ref{HoHu89}.
\end{proof}

We are now ready to prove Theorem \ref{MainThm}. We proceed by reducing to the $F$-finite case via gamma construction.
This method is well-know to the experts. We refer to  \cite[Dicussion 6.11 \& Lemma 6.13]{HoHu2}  for definitions and properties. 
\begin{proof}[Proof of Theorem \ref{MainThm}]  
Completion does not change the assumption that $R/I$ is $F$-rational by Remark \ref{completion}.
In addition, if $\widehat R$ is strongly $F$-regular, then so is $R$,
because $\tau(R) \widehat{R}=\tau(\widehat{R})$ \cite[Theorem $2.3$]{LyuKaren}. 
We now consider a $p$-base for $K^{1/p}$, $\Lambda$, and a cofinite set 
$\Gamma\subseteq \Lambda$.
Consider the faithfully flat extension $R \to \GG R$ given by the gamma construction. We have that $\GG R$ is a complete $F$-finite local ring with maximal ideal $\GG \m := \m \GG R$ and residue field $\GG K \cong K \otimes_R \GG R$. Since $R$ is complete, there exists a Gorenstein local ring $(S,\n,K)$, with $\dim(R) = \dim(S)$, such that $R$ is a homomorphic image of $S$. By functoriality of the Gamma construction, we have that $\GG S$ maps homomorphically onto $\GG R$. Furthermore, the map $S \to \GG S$ is local and flat, with $\GG K$ as closed fiber. Since $S$ is a Gorenstein ring,  so is $\GG S$. 
Because $I$ is a canonical module of $R$, we have that $I \cong \Hom_S(R,S)$. Since $S \to \GG S$ is faithfully flat, we have
\[
\ds I \otimes_R \GG R \cong \Hom_S(R,S) \otimes_S \GG S \cong \Hom_{\GG S}(R \otimes_S \GG S, \GG S) \cong \Hom_{\GG S}(\GG R,\GG S).
\]
Therefore, $I \otimes_R \GG R \cong \omega_{\GG R}$ is a canonical module for $\GG R$. In addition, $I\otimes_R \GG R \cong I\GG R$; therefore, $\GG I:= I \GG R \subseteq \GG R$ is a canonical ideal of $\GG R$. Now consider the flat local homomorphism $R/I \to \GG R/\GG I$ induced by the 
gamma construction above. We have that $R/I$ is excellent because it is complete. In addition, $R/I$ is a Cohen-Macaulay domain because it is $F$-rational. Furthermore, the closed fiber is $\GG K$, and $R/I$ is Gorenstein by Lemma \ref{dual}. 
We have that $\GG R/\GG I$ is $F$-rational because parameter ideals are tightly closed in $\GG R/ \GG I$ \cite[Proposition 3.2]{Aberbach01}.
Since $\GG R$ is $F$-finite,  we conclude that $\GG R$ is strongly $F$-regular by Theorem \ref{MainTech}. In this case, $\tau(\GG R) = \GG R$, so that $0^*_{E_{\GG R}(\GG K)} = 0$. Now, we apply \cite[Corollary 3.8]{AEBaseChange} to get that $\tau(R)\GG R = \tau(\GG R) = \GG R$. Therefore, $\tau(R) = R$, and so, $R$ is strongly $F$-regular. 
\end{proof}

As a corollary of this result, we weaken the Cohen-Macaulay assumption in Proposition \ref{Goto}  to $S_2$ for  excellent rings of positive characteristic.

\begin{corollary}\label{ExtGoto}
Let $(R,\m,K)$ be an  excellent local ring  of positive characteristic. Suppose that $R$ is $S_2$, and it has a canonical ideal $I\subsetneq R$ such that $R/I$ is regular. Then $R$ is regular.
\end{corollary}
\begin{proof} As $R/I$ is regular,  it is an $F$-rational ring. By Theorem \ref{MainThm} $R$ is strongly $F$-regular; therefore, $R$ is Cohen-Macaulay. We can now apply Proposition \ref{Goto} to get the desired result.
\end{proof}

Finally, we give an example which shows that the sufficient condition for strong $F$-regularity in Theorem \ref{MainTech} is not necessary.
That is, if $R$ is strongly $F$-regular, we cannot always find a canonical ideal $I \subseteq R$ such that $H^{d-1}_\m(R/I)$ is a simple $R\{F\}$-module, or equivalently such that $R/I$ is strongly $F$-regular. 
\begin{example} Let $K$ be a field with char$(K) \ne 3$, and consider the two dimensional domain
\[
R = K\ps{s^3,s^2t,st^2,t^3} \cong K\ps{x,y,z,w}/J,
\] 
where $J = (z^2-yw,yz-xw,y^2-xz)$. Then $R$ has type two, and it is a direct summand of $k\ps{s,t}$; therefore, it is strongly $F$-regular \cite[Theorem 3.1 (e)]{HoHuStrong}. Any canonical ideal of $R$ is two-generated, say $I = (f,g)$. Denote by $\m$ the maximal ideal of $R$. If $R/I$ was strongly $F$-regular, then it would be regular because it is a one-dimensional ring. This would be equivalent to $\dim_K(\m/(\m^2+I+J)) = 1$. Since $J \subseteq \m^2$, we have
\[
\ds \dim_K\left(\frac{\m}{\m^2+I+J}\right) = \dim_K \left(\frac{\m}{\m^2+(f,g)}\right) \gs 2.
\]
\end{example}

\begin{remark}\label{RefGeometry}  
If $R$ is an $F$-finite normal domain, one can use a more geometric argument to show that if $R/I$ is strongly F-regular, then so is $R$. This is done by choosing a canonical divisor on $\Spec(R)$ corresponding to the ideal $I\cong \omega_R$, and by using F-adjunction \cite{FAdj}.
\end{remark}

\section*{Acknowledgments} 
We would like to thank Linquan Ma for valuable comments which greatly improved the exposition of the paper and for suggesting to us the one-dimensional version of Proposition \ref{Goto}. 
We thank Karl Schwede for pointing out Remark \ref{RefGeometry}.
We also thank Craig Huneke for useful discussions. 
We thank the referee for helpful comments. In particular, for suggesting to drop the assumption of $F$-finite in Theorem \ref{MainThm}.
The first author was partially supported by NSF Grant DMS-$1259142.$
The second author thanks the National Council of Science and Technology of Mexico (CONACyT) for support through Grant \#207063.

\bibliographystyle{alpha}
\bibliography{References}
{\footnotesize

\noindent \small \textsc{Department of Mathematics, University of Virginia, Charlottesville, VA  22903} \\ \indent \emph{Email address}:  {\tt ad9fa@virginia.edu} 
}

\vspace{.25cm}

\noindent \small \textsc{Department of Mathematics, University of Virginia, Charlottesville, VA  22903} \\ \indent \emph{Email address}:  {\tt lcn8m@virginia.edu} 

\end{document}